\documentclass[a4paper,12pt,reqno]{amsart}

 \usepackage[T2A]{fontenc}
 \usepackage[utf8]{inputenc}
 \usepackage[ukrainian,english]{babel}
 \usepackage{amssymb,upref}
\usepackage{amsthm}

\usepackage{graphicx}
\usepackage{mathrsfs}

\usepackage[text={130mm,200mm}]{geometry}

\theoremstyle{plain}
\newtheorem{theorem}{Theorem}
\newtheorem*{theorem*}{Theorem}
\newtheorem{corollary}{Corollary}
\newtheorem*{corollary*}{Corollary}
\newtheorem{lemma}{Lemma}
\newtheorem*{lemma*}{Lemma}

\newtheorem*{proposition*}{Proposition}

\newtheorem*{conjecture*}{Conjecture}
\theoremstyle{definition}
\newtheorem{definition}{Definition}
\newtheorem*{definition*}{Definition}
\theoremstyle{remark}
\newtheorem{remark}{Remark}
\newtheorem*{remark*}{Remark}

\begin{document}

\title[One continuum class of fractal functions]{One continuum class of fractal functions defined in terms of $Q^*_s$-representation}

\author{V. V. Nazarchuk}
\address[V. V. Nazarchuk]{Institute of Mathematics of NAS of Ukraine, Kyiv, Ukraine\\
ORCID 0009-0009-7584-6370}
\email{ nazarchukvalentyna@imath.kiev.ua}

\author{S. O. Vaskevych}
\address[S. O. Vaskevych]{Institute of Mathematics of NAS of Ukraine, Kyiv, Ukraine\\
ORCID 0009-0005-3979-4543}
\email{svetaklymchuk@imath.kiev.ua}

\author{S. P. Ratushniak}
\address[S. P. Ratushniak]{Institute of Mathematics of NAS of Ukraine,  Dragomanov Ukrainian State University, Kyiv, Ukraine\\
ORCID 0009-0005-2849-6233}
\email{ratush404@gmail.com}



\thanks{Bukovinian Math. Journal. 2024, 12, 2, 154–161}
\thanks{This work was supported by a grant from the Simons Foundation (1030291, V.N.)}

\begin{abstract}
In the paper we study a class $F$ of multiparameter functions defined 
in terms of a polybasic $s$-adic $Q^{*}_{s}$-representation of numbers by
\begin{equation*}
f_a\bigl(x=\Delta^{Q^{*}_s}_{\alpha_1\alpha_2\ldots\alpha_n\ldots}\bigr)
=
\Delta^{Q^{*}s}_{|a_1-\alpha_1|\,|a_2-\alpha_2|\,\ldots\,|a_n-\alpha_n|\ldots},
\end{equation*}
where $(a_n)$ is the sequence of digits for $s$-adic representation of the parameter 
$a\in[0,1]$, and
\begin{equation*}
\Delta^{Q^{*}_s}_{\alpha_1\alpha_2\ldots\alpha_n\ldots}=
\beta_{\alpha_1 1}+
\sum_{n=2}^{\infty} \left(
\beta_{\alpha_n n} \prod_{j=1}^{n-1} q_{\alpha_j j}
\right)
\end{equation*}
is the $Q^{*}_{s}$-representation of real numbers generated by a positive stochastic 
matrix $\|q_{ij}\|$ with $\beta_{\alpha_n n}=\sum\limits_{i=0}^{\alpha_n-1} q_{in}$.
For a fixed $Q^{*}_{s}$-representation of numbers the function $f_a$ is determined by 
the parameter $a$, which makes the class of functions continuum. 
In this paper we investigate the continuity of the function $f_a$ on the sets of $Q^{*}_{s}$-binary and $Q^{*}_{s}$-unary numbers. We prove that the functions in this class are continuous on the set of numbers with a unique $Q^{*}_{s}$-representation. Furthermore, we show that except for $f_0$ and $f_1$, all functions have a countable set of discontinuities at $Q^{*}_{s}$-binary points. 
We classify the topological types of the value sets of $f_a$ depending on the parameter $a$. We prove that, if the value set is of Cantor type, then it is zero-dimensional. These properties reveal the fractal nature of functions in the class $F$. 
We describe the structural properties of the level sets of $f_a$ in terms of the digits of the $s$-adic representation of $a$. In particular, we establish that a level set of the function $f_a$ can be an empty set, a finite set, or a continuum.  For certain values of $s$ we provide examples of fractal level sets and calculate its fractal dimensions.
\end{abstract}

\maketitle
\section*{Introduction}

We call a function fractal if its range, graph [2, 3, 7], level sets [4], 
or other sets associated with the function possess fractal properties [1, 5]. 
Among functions defined on the interval $[0,1]$ that exhibit complicated local
structure and fractal properties, particular attention deserves the class of
functions that have a countable set of discontinuities and are continuous at
all remaining points. This paper studies functions of this type.
To define and study these functions analytically, we use the polybasic
representation of real numbers (the $Q^{*}_{s}$-representation), which depends
on infinitely many parameters and generalizes the classical $s$-adic
representation.
Throughout the paper we fix the $Q^{*}_{s}$-representation and consider a
one-parameter class of functions with parameter
$a \in [0,1]$.
In the paper we investigate topological and metric properties of the range
and the level sets of functions from this class.
It is worth noting that this class contains two continuous functions: the
identity mapping of the unit interval and the digit inversion of the
$Q^{*}_{s}$-representation of a number.

\section{Polybasic $Q^{*}_{s}$-representation of real numbers}

Let $A\equiv \{0,1,\ldots ,s-1\}$ be an $s$-adic alphabet and let $L=A\times A\times \ldots$ denote the space of infinite sequences whose elements belong to the alphabet $A$.
Let $\|q_{in}\|$ be a stochastic matrix such that $0<q_{in}<1$ and $q_{0n}+q_{1n}+\ldots+q_{s-1n}=1$ for $n\in N$ and
\begin{equation}
\prod_{n=1}^{\infty}\max_{i\in A}\{q_{in}\}=0.
\end{equation}

Then  [6]  for any $x \in [0,1]$, there exists a sequence 
$(\alpha_n) \in L$ such that
\begin{equation}
x=\beta_{\alpha_1 1}
+\sum_{n=2}^{\infty}
\left(
\beta_{\alpha_n n}
\prod_{j=1}^{n-1} q_{\alpha_j j}
\right)
\equiv
\Delta^{Q^{*}_s}_{\alpha_1\alpha_2\ldots\alpha_n\ldots},
\end{equation}
where $\beta_{\alpha_n n}=\sum\limits_{i=0}^{\alpha_n-1} q_{in}$ (i.e. $\beta_{0n}=0$, $\beta_{1n}=q_{0n}$, $\beta_{2n}=q_{0n}+q_{1n}$, $\dots$, $\beta_{s-1,n}=1-q_{s-1n}$).
Expansion (2) defines the $Q^{*}_{s}$-expansion of the number $x$ and the abbreviated notation $\Delta^{Q^{*}_s}_{\alpha_1\alpha_2\ldots\alpha_n\ldots}$
denotes the $Q^{*}_{s}$-representation of $x$.

If $q_{in}=q_i$ for all $n\in N$ and $i\in A$, then the $Q^{*}_{s}$-representation becomes the self-similar $Q_{s}$-representation and if $q_i=\frac{1}{s}$ for $i\in A$, then the $Q^{*}_{s}$-representation coincides with the classical $s$-adic representation.

There exist numbers that have two different $Q^{*}_{s}$-representations. 
These are the numbers with representations
\begin{equation}
\Delta^{Q^{*}_s}_{\alpha_1\alpha_2\ldots\alpha_{m-1}\alpha_m(0)}
=
\Delta^{Q^{*}_s}_{\alpha_1\alpha_2\ldots\alpha_{m-1}[\alpha_m-1](s-1)}.
\end{equation}
Such numbers are called \emph{$Q^{*}_{s}$-binary numbers}. The set of these numbers is countable.
All remaining numbers of the unit interval have a unique representation and are called \emph{$Q^{*}_{s}$-unary numbers}.

\begin{definition} 
The \emph{cylinder of rank $m$ with base $c_1c_2\dots c_m$} is the set
\[
\Delta^{Q^{*}_s}_{c_1c_2\dots c_m} =
\Biggl\{
x : x = 
\sum_{k=1}^{m} \beta_{c_k k} \prod_{j=1}^{k-1} q_{c_j j}
+
\prod_{i=1}^{m} q_{c_i i} \cdot \sum_{n=m+1}^{\infty} \beta_{\alpha_n n} \prod_{j=m+1}^{n-1} q_{\alpha_j j}
\Biggr\}.
\]
\end{definition}
Cylinders of $Q^{*}_{s}$-representation satisfy the following properties for all 
$(c_1,\dots,c_m)$ and all $m\in N$:
\begin{enumerate}
\item $\Delta^{Q^{*}_s}_{c_1c_2\dots c_m} =
\bigcup\limits_{i=0}^{s-1} \Delta^{Q^{*}_s}_{c_1c_2\dots c_m i}$;

\item $\Delta^{Q^{*}_s}_{c_1c_2\dots c_m} = [a,b], \quad
a = \sum\limits_{i=1}^{m} \beta_{c_i i} \prod\limits_{j=1}^{i-1} q_{c_j j}, \quad
b = a + \prod\limits_{i=1}^{m} q_{c_i i}$;

\item $\bigcap\limits_{m=1}^{\infty} \Delta^{Q^{*}_s}_{c_1c_2\dots c_m} 
= \Delta^{Q^{*}_s}_{c_1c_2\dots c_m\dots} = x$;

\item $|\Delta^{Q^{*}_s}_{c_1c_2\dots c_m}| = \prod\limits_{i=1}^{m} q_{c_i i} = q_{c_m m} |\Delta^{Q^{*}_s}_{c_1c_2\dots c_{m-1}}|, 
\quad |\Delta^{Q^{*}_s}_{c_1}| = q_{c_1 1}|[0,1]|$.
\end{enumerate}

\begin{remark}
Now we consider a $Q^{*}_{s}$-representation that satisfies the conditions
\[
0 < \varepsilon < \min_i \{q_{in}\} \quad \text{and} \quad 
\max_i \{q_{in}\} < \delta < 1, \quad \forall n \in \mathbb{N},
\]
which guarantee that the corresponding Cantor-type sets 
\[
C[Q^{*}_s, V_n] = \{x : x = \Delta^{Q^{*}_s}_{\alpha_1\alpha_2\ldots\alpha_n\ldots}, \ \alpha_n \in V_n \subset A\},
\]
are zero-dimensional when $V_n \neq A$ infinitely many times [6].
\end{remark}

\section{Central object of study}

Let $a$ be a fixed number in the unit interval $[0,1]$ with $s$-adic representation
\[
a = \Delta^s_{a_1 a_2 \ldots a_n \ldots} 
= \frac{a_1}{s} + \frac{a_2}{s^2} + \dots + \frac{a_n}{s^n} + \dots, \quad \text{  where } (a_n) \in L.
\]
The main object of our study is the function $f_a$ defined on $[0,1]$ by equality
\begin{equation}
f_a\bigl(x = \Delta^{Q^{*}_s}_{\alpha_1 \alpha_2 \ldots \alpha_n \ldots} \bigr)
=
\Delta^{Q^{*}_s}_{|a_1 - \alpha_1| \, |a_2 - \alpha_2| \ldots |a_n - \alpha_n| \ldots}
\end{equation}
Clearly, $a$ is one of the parameters that determines the function $f_a$.
We denote the class of such functions by $F$.

The simplest representatives of the class $F$ are the functions
\begin{equation}
f_a\!\left(x=\Delta^{Q^{*}_s}_{\alpha_1\alpha_2\alpha_3\ldots\alpha_n\ldots}\right)
=
\Delta^{Q^{*}_s}_{[s-1-\alpha_1][s-1-\alpha_2][s-1-\alpha_3]\ldots[s-1-\alpha_n]\ldots},\quad a=\Delta^{Q^*_s}_{(s-1)}.
\end{equation}

The function defined by equality~(5) is called the \emph{digit inversor of the
$Q^{*}_{s}$-representation}. It is a continuous, strictly decreasing, singular
function whenever $q_{in} \ne q_{[s-1-i]n}$ for all $ n\in N$ and $i\in A$.
\begin{equation}
f_a\!\left(x=\Delta^{Q^{*}_{s}}_{\alpha_1\alpha_2\alpha_3\ldots\alpha_{2k-1}\alpha_{2k}\ldots}\right)
=
\Delta^{Q^{*}_{s}}_{
[s-1-\alpha_1]\alpha_2[s-1-\alpha_3]\ldots[s-1-\alpha_{2k-1}]\alpha_{2k}\ldots}, \quad
a=\Delta^{Q^*_s}_{([s-1]0)}.
\end{equation}

Since the inequality
\[
f_a(\Delta^{Q^{*}_{s}}_{\alpha_1 ... \alpha_n(0)})
=
\Delta^{Q^{*}_{s}}_{|a_1-\alpha_1|...|a_n-\alpha_n|a_{n+1}...}
\neq
f_a(\Delta^{Q^{*}_{s}}_{\alpha_1...[\alpha_n-1](s-1)})
=
\Delta^{Q^{*}_{s}}_{|a_1-\alpha_1|...|a_n-\alpha_n+1||a_{n+1}-s+1|...}
\]
holds, the definition of most functions in the class $F$ requires a
convention that uses only one of the two $Q^{*}_{s}$-representations of
$Q^{*}_{s}$-binary numbers. We adopt the representation that contains period $(0)$.

Note that the values of the function $f_a$ corresponding to two different
representations of the same number coincide only in the case when
$a_{n+i}=s-1-|a_{n+i}-s+1|$ for each $i\in N\cup\{0\}$, that is, when
$a_{n+i}=s-1$ or  $a_{n+i}=0$ for $n\in N$.

\begin{lemma}
For numbers $a=\Delta^{Q^*_s}_{a_1a_2\ldots a_n\ldots}$ and $b=\Delta^{Q^*_s}_{b_1b_2\ldots b_n\ldots}$ such that the sequence
$\left(\frac{a_n+b_n}{2}\right)\in L$ the following equality holds:
\[
f_a\!\left(x_0=\Delta^{Q^{*}_{s}}_{c_1c_2\ldots c_n\ldots}\right)
=
f_b\!\left(x_0=\Delta^{Q^{*}_{s}}_{c_1c_2\ldots c_n\ldots}\right),\quad \text{where } c_n=\frac{a_n+b_n}{2}.
\]
\end{lemma}
\begin{proof}
Let a number $a=\Delta^{Q^{*}_{s}}_{a_1a_2\ldots a_n\ldots}$
be given. Choose a number $b\in[0,1]$ such that $b=\Delta^{Q^{*}_{s}}_{b_1b_2\ldots b_n\ldots}$ and $\frac{a_n+b_n}{2}\in A$ for all $n\in N$ (obviously, such a number exists). Then\\
\begin{align*}
  f_a\left(\Delta^{Q^{*}_{s}}_{c_1\ldots c_n\ldots}\right) &= \Delta^{Q^{*}_{s}}_{\frac{|a_1-b_1|}{2}\ldots
\frac{|a_n-b_n|}{2}\ldots}
=  \Delta^{Q^{*}_{s}}_{\frac{|b_1-a_1|}{2}\ldots
\frac{|b_n-a_n|}{2}\ldots} = \\
  & = \Delta^{Q^{*}_{s}}_{\frac{|2b_1-(a_1+b_1)|}{2}\ldots
\frac{|2b_n-(a_n+b_n)|}{2}\ldots}
= \Delta^{Q^{*}_{s}}_{\left|b_1-\frac{a_1+b_1}{2}\right|
\ldots
\left|b_n-\frac{a_n+b_n}{2}\right|
\ldots}
=
f_b(x_0). 
\end{align*}
\end{proof}

\begin{corollary}
If $a=\Delta^s_{[2d_1][2d_2]\ldots[2d_n]\ldots}$, then
$f_a\!\left(x=\Delta^{Q^{*}_{s}}_{d_1 d_2 \ldots d_n \ldots}\right)=x$.

Indeed, for $a=\Delta^s_{[2d_1]\ldots[2d_n]\ldots}$ we obtain
$f_a\!\left(x=\Delta^{Q^{*}_{s}}_{d_1 \ldots d_n \ldots}\right)=
\Delta^{Q^{*}_{s}}_{|2d_1-d_1|\ldots|2d_n-d_n|\ldots} =x$.
\end{corollary}

\begin{theorem}
Functions $f_a$ of the class $F$ are continuous on the set of
$Q^{*}_{s}$-unary numbers, whereas on the set of $Q^{*}_{s}$-binary numbers
the functions $f_a$ are continuous only for $a=0$ and $a=1$.
\end{theorem}
\begin{proof}
Let $a=\Delta^{s}_{a_1a_2\ldots a_n\ldots}$. We prove the continuity of the corresponding function $f_a$ at a $Q^{*}_{s}$-unary point. Let $x_0=\Delta^{Q^{*}_{s}}_{\alpha_1\alpha_2\ldots\alpha'_n\ldots}$
be a $Q^{*}_{s}$-unary point. 
Consider a point $x$ such that $x\neq x_0$. Then there exists an index $n$
for which $\alpha'_n=\alpha_n(x_0)\neq \alpha_n(x)$ while $\alpha_k(x)=\alpha_k(x_0)$ for $k<n$.
The condition $n\to\infty$ is equivalent to $x\to x_0$.
By the definition of continuity of the function $f$ at the point $x_0$, that is,
$\lim\limits_{x\to x_0} f(x)=f(x_0)$, we obtain $\lim\limits_{x\to x_0}\lvert f(x)-f(x_0)\rvert =0$.
We show that
\[
\lim_{x\to x_0}\lvert f_a(x)-f_a(x_0)\rvert =0.
\]

Consider the expression
\begin{align*}
|f_a(x)-f_a(x_0)|
&=
\left|
f_a\!\left(\Delta^{Q^{*}_{s}}_{\alpha_1\ldots\alpha_{n-1}\alpha_n\ldots}\right)
-
f_a\!\left(\Delta^{Q^{*}_{s}}_{\alpha_1\ldots\alpha_{n-1}\alpha'_n\ldots}\right)
\right| \\
&=
\Bigl|
\Delta^{Q^{*}_{s}}_{|a_1-\alpha_1|\ldots|a_{n-1}-\alpha_{n-1}||a_n-\alpha_n|\ldots}
-
\Delta^{Q^{*}_{s}}_{|a_1-\alpha_1|\ldots|a_{n-1}-\alpha_{n-1}||a_n-\alpha'_n|\ldots}
\Bigr| \\
&=
\prod_{i=1}^{n-1} q_{\alpha_i i}\cdot W, \quad \text{where } 0<W\le 1.
\end{align*}
Then
\[
\lim_{x\to x_0}\lvert f_a(x)-f_a(x_0)\rvert
\le
\lim_{n\to\infty}\prod_{i=1}^{n-1} q_{\alpha_i i}
\le
\lim_{n\to\infty}\prod_{i=1}^{n-1}
\max_{\alpha_i\in A}\{q_{\alpha_i i}\}
=0.
\]
Therefore, $\lim\limits_{x\to x_0} f_a(x)=f_a(x_0)$ at every $Q^{*}_{s}$-unary point.

According to the adopted convention, the value of the function $f_a$ at a $Q^{*}_s$-binary point (3) is determined using the first $Q^{*}_s$-representation and evaluated via formula (4). The function exhibits an essential (non-removable) discontinuity at that point.
\end{proof}

\section{Structural properties of the functions in the class $F$}
\begin{theorem}
The value set $E_{f_a}$ of the function $f_a$, where $a = \Delta^s_{a_1 a_2 \ldots a_n \ldots}$, is given by
\[
C[Q^{*}_s; V_n] =\{ x \in [0,1] : x = \Delta^{Q^{*}_s}_{\alpha_1  \ldots \alpha_n \ldots}, 
\alpha_n \in V_n \equiv \{0,1,\dots, \max(s-1-a_n, a_n)\}, \ n\in N \}.
\]
\end{theorem}
\begin{proof}
Let the function $f_a$ be given, with $a = \Delta^s_{a_1 a_2 \ldots a_n \ldots}$.
Then the value of the function $f_a$ has the representation 
$\Delta^{Q^{*}_s}_{|a_1-\alpha_1| \ldots |a_n-\alpha_n| \ldots}$ with
$(\alpha_n) \in L$. 
It is clear that the $n$-th digit of the $Q^{*}_s$-representation of a number $y$ 
can take values from $V_n = \{ a_n, |a_n-1|, \ldots, |a_n - s + 1| \}$.
Since $a_n \in A$, then the elements of $V_n$ belong to the sequence of digits of 
the alphabet $A$, except for those whose sum with $a_n$ exceeds $(s-1)$, 
that is, the digits from $0$ to the larger of $(s-1-a_n)$ and $a_n$.
Thus, the set of possible values of an arbitrary digit in the $Q^{*}_s$-representation
of $y$ determines the value set of the function $f_a$.
\end{proof}

\begin{corollary}
The value set of the function $f_a$ is a subset of the interval
$[0, \Delta^{Q^{*}_s}_{c_1 c_2 \dots c_n \dots} ]$, where
$c_n = \max\{s-1-a_n, \, a_n\}$,  $n \in N$.
\end{corollary}

\begin{theorem}
Let $s>2$. The value set $E_{f_a}$ of the function $f_a$, where
$a = \Delta^s_{a_1 a_2 \ldots a_n \ldots}$ is a union of intervals if $a_n \in \{1,2,\dots,s-2\}$ for only finitely many indices $n$, 
or a continuous nowhere dense zero-measure set if $a_n \in \{1,2,\dots,s-2\}$ for infinitely many indices $n$.
\end{theorem}

\begin{proof}
To prove the continuity of the value set $E_{f_a}$, we show that for an arbitrary
digit in the $Q^{*}_s$-representation, the set of possible function values
$V_n = \{0, 1, \ldots, \max\{s-1-a_n, a_n\}\}$ for $a_n \in A_s$,
contains at least two elements.
Since the elements $(s-1-a_n)$ and $a_n$ are inverse in the $s$-adic alphabet, then
the minimal value that $\max\{s-1-a_n, a_n\}$ can take equal $\frac{s}{2}$ if $s$ is even, 
or $\frac{s-1}{2}$ if $s$ is odd. 
Thus, for any $n$, the set $V_n$ contains both $0$ and $1$ for any $s>2$, 
so each digit in the $Q^{*}_s$-representation of a function value has at least two alternatives.
Since the quantity of digits in the representation of $y$ is countable, then the set of distinct sequences of digits specifying the function values constitutes a continuum.
Hence, the value set $E_{f_a}$ is continuum.

Let there exist $k$ indices $n_1, n_2, \ldots, n_k$ for which 
$a_{n_i} \in \{1,2,\ldots,s-2\}$ for all $i=1,\ldots,k$. 
For simplicity, assume that $a_{n_i} = 1$ for all $i=1,\dots,k$. 
Then $V_{n_i} = \{0,1,\ldots,s-2\}$, and therefore
\[
E_{f_a} = [0,1] \setminus 
\left(\bigcup_{r_{n_1} \neq s-1} \bigcup_{ r_{n_2} \neq s-1}  \ldots \bigcup_{r_{n_k} \neq s-1}
\Delta^{Q^{*}_s}_{r_1 r_2 \ldots r_m}\right).
\]
Since a cylinder of $Q^{*}_s$-representation is an interval, then after removing from the unit interval $s^k$ non-overlapping subintervals (cylinders of the form 
$\Delta^{Q^{*}_s}_{r_1 r_2 \ldots r_{i-1}[s-1]}$ for $i \in \{n_1, n_2, \ldots, n_k\}$),
we see that the set $E_{f_a}$ is a finite union of $(n_k+1)$-th rank cylinders. 
Hence, it is a finite union of intervals.

Let there exist infinitely many indices for which $a_{n_i} \in \{1,2,\ldots,s-2\}$ for all $i \in N$. 
Consider an arbitrary cylinder $\Delta^{Q^{*}_s}_{r_1 r_2 \ldots r_m}$.
There exist cylinders $\Delta^{Q^{*}_s}_{r_1 r_2 \ldots r_m \ldots r_{m+j}} \subset \Delta^{Q^{*}_s}_{r_1 r_2 \ldots r_m}$
that contain no points of the set $E_{f_a}$. For example, such cylinders are 
$\Delta^{Q^{*}_s}_{r_1 r_2 \ldots r_m \ldots r_{m+j-1}0}$ and $\Delta^{Q^{*}_s}_{r_1 r_2 \ldots r_m \ldots r_{m+j-1}[s-1]}$.
Since $m$ is an arbitrary natural number, then for any interval whose endpoints belong to $E_{f_a}$, there exists a whole interval of points not belonging to $E_{f_a}$. 
Hence, $E_{f_a}$ is a nowhere dense set.

To prove the zero measure of 
$C[Q^{*}_s; V_n]$, we show that its measure equals zero under the condition 
$q_{[s-1] n} = \min\limits_i \{q_{i n}\}$
which does not affect the generality of the argument. 
Consider the set
\[
C[Q^{*}_s; V_n] = \Bigl\{ x \in [0,1] : x = \Delta^{Q^{*}_s}_{\alpha_1 \alpha_2 \ldots \alpha_n \ldots}, \ 
\alpha_n \in V_n \equiv \{0,1,\dots,s-2\}, \ \forall n \in N \Bigr\}.
\]
The set $C[Q^{*}_s; V_n]$ can be represented as
\[
C[Q^{*}_s; V_n] = [0,1] \setminus 
\bigcup_{m=1}^{\infty} \left(
\bigcup_{c_1 \in V_1} \ldots \bigcup_{c_m \in V_m} 
\Delta^{Q^{*}_s}_{c_1 c_2 \ldots c_m [s-1]}\right).
\]
Then, by the additive property of the Lebesgue measure we have
\[
\lambda\left(C[Q^{*}_s; V_n]\right) 
= 1 - \sum_{m=1}^{\infty} 
\left( \prod_{i=1}^{m-1} q_{c_i i} \, q_{[s-1] m} \right) = 0.
\]
Since the Lebesgue measure of the ``largest'' set equals zero, it follows that the Lebesgue measure of the value set for any $a$ such that $a_n \in \{1,2,\dots,s-2\}$ for infinitely many indices $n$ is also zero.
\end{proof}

\begin{remark}
  If $s = 2$, then $E_{f_a} = [0,1]$.
\end{remark}

Recall that the level set of a function is defined as
\[
f_a^{-1}(y_0) = \{ x \in [0,1] : f_a(x) = y_0 \}.
\]

\textbf{Example 1.} Let $a = \Delta_5 (314)$. Then the level set $y_0 = \Delta^{Q^{*}_5}_{4(0)}$ 
is empty set; 
the level set $y_0 = \Delta^{Q^{*}_5}_{(0)}$ contains a single point
$f_a^{-1}(\Delta^{Q^{*}_5}_{(0)}) = \{ a = \Delta^{Q^{*}_5}_{(314)} \}$; 
the level set $y_0 = \Delta^{Q^{*}_5}_{\underbrace{110\dots110}_{3n}(0)}$
contains $4^n$ points;  
the level set $y_0 = \Delta^{Q^{*}_5}_{(110)}$ is a continuum set.

\begin{theorem}
Let $s > 2$. The level set $f_a^{-1}(y_0) = f_a^{-1}(\Delta^{Q^{*}_s}_{b_1 b_2 \ldots b_n \ldots})$ of the function $f_a$ generated by the parameter 
$a = \Delta^s_{a_1 a_2 \ldots a_n \ldots}$
has the following structure:

1) it is empty if $b_n = s-1$ and $a_n \in \{1,2,\dots,s-2\}$;  

2) it is finite if only for a finite number of indices $n$ we have either 
\[
\begin{cases}
  a_n + b_n \in A, \\
  a_n - b_n \in A;
\end{cases}
 \text{or} \quad b_n = 0, \quad \forall n \in N;
\]  

3) it is a continuum if for infinitely many indices $n$ we have
\[
\begin{cases}
  b_n \neq 0,  \\
  a_n + b_n \in A, \\
  a_n - b_n \in A.
\end{cases}
\]
\end{theorem}
\begin{proof}
Let a number $a = \Delta^s_{a_1 a_2 \ldots a_n \ldots}$ and the function $f_a$ generated by it be given. Consider $y_0 = \Delta^{Q^{*}_s}_{b_1 b_2 \ldots b_n \ldots}$. 
Then the level set of $y_0$ consists of the solutions of the equation
$f_a\left(\Delta^{Q^{*}_s}_{\alpha_1(x) \alpha_2(x) \ldots \alpha_n(x) \ldots} \right) = \Delta^{Q^{*}_s}_{b_1 b_2 \ldots b_n \ldots}$.
By the definition of the function, we obtain the following system of equations:
\begin{equation}
\begin{cases}
|a_1 - \alpha_1| = b_1, \\
\cdots \\
|a_n - \alpha_n| = b_n, \\
\cdots
\end{cases}
\end{equation}
Since $b_n \in A$, then the equation $|a_n - \alpha_n| = b_n$ 
is equivalent to the set of equation $a_n - \alpha_n = b_n$ and $a_n - \alpha_n = -b_n$,
i.e., $\alpha_n = a_n - b_n$  or $\alpha_n = a_n + b_n$.
If both $a_n - b_n$ and $a_n + b_n$ belong to $A$ simultaneously, then the equation 
$|a_n - \alpha_n| = b_n$ has at most two solutions.  
Moreover, if this occurs for infinitely many equations in system (7) and $b_n \neq 0$ infinitely often (in this case the equation $|a_n - \alpha_n| = b_n$ has two solutions), then the system has a continuum of solutions.

If only one of the values $\alpha_n = a_n - b_n$ or $\alpha_n = a_n + b_n$ belongs to $A$, or if $b_n = 0$ for all $n \in N$, then the equation has a unique solution.  
Moreover, if only a finite number of equations have two solutions while all others have one, then the system (7) has a finite number of solutions.  
Finally, if for the equation $|a_n - \alpha_n| = b_n$ we have $b_n = s-1$ and $a_n \neq s-1$ or $a_n \neq 0$, then the equation has no solutions, and consequently the system 
(7) also has none.

The quantity of solutions of system (7) corresponds to the quantity of preimages 
of the level $y_0$, which completes the proof of the theorem.
\end{proof}

\section{Some partial cases}

Let the $Q^{*}_s$-representation be a $Q_3$-representation with parameters $(q_0, q_1, q_2)$.

\textbf{Example 1.} If $a_n = 1$ for all $n \in N$, then:

1) the value set of the function is the self-similar Cantor-type set $C[Q_3, \{0,1\}]$ 
with fractal dimension given as the solution of the equation $q_0^x + q_1^x = 1$;

2) the level set $y_0 = \Delta^{Q_3}_{(1)}$ is the self-similar Cantor-type set 
$C[Q_3, \{0,2\}]$, whose fractal dimension is the solution of the equation
$q_0^x + q_2^x = 1$.

Note that in this case the fractal properties of all level sets of the function can be easily determined.

\end{document}